\documentclass{article}
\usepackage{graphicx}
\usepackage{amsmath}
\usepackage{amssymb}
\usepackage{amsthm}
\usepackage{amsfonts}
\usepackage{mathrsfs}
\usepackage{verbatim}
\usepackage{stmaryrd}
\usepackage{wasysym}
\def\Underline{\setbox0\hbox\bgroup\let\\\endUnderline}
\def\endUnderline{\vphantom{y}\egroup\smash{\underline{\box0}}\\}
\def\|{\verb|}

\newtheorem{theorem}{Theorem}
\newtheorem{lemma}{Lemma}
\newtheorem{corollary}{Corollary}
\newtheorem{definition}{Definition}

\newtheorem{axiom}{Axiom}

\newcommand{\nplay}{\widetilde{\mathbb{N}{\rm p}}}
\newcommand{\afimp}{\widetilde{\mathbb{I}{\rm m}}}

\newcommand{\limp}{\widetilde{\mathbb{LI}{\rm m}}}
\newcommand{\ims}{\widetilde{\mathbb{I}{\rm ms}}}

\newcommand{\elimp}{\langle \ims^\infty \cup \limp \rangle}
\newcommand{\nim}{\mathbb{N}{\rm im}}
\newcommand{\moon}{\rightmoon}
\newcommand{\minfty}{\overline{\infty}}
\def\conj#1{\overleftrightarrow{#1}}
\def\lopt#1{#1^{\mathrm{L}}}
\def\ropt#1{#1^{\mathrm{R}}}

\def \sglstar{*}
\def \nimber#1{* #1}
\def \eqeq{\asymp}
\def \mex{{\rm mex}}

\def \spmoon{\rrparenthesis}

\begin{document}

\title{What happens when we add impartial loopy games and impartial entailing games?}

\author{Koki Suetsugu}

\maketitle

\begin{abstract}
The disjunctive sum of impartial games is analyzed by Sprague-Grundy theory. The theory has been extended to loopy games and entailing games by early results. In this study, we consider further extension of this theory and show partial algebraic structure of the sum of loopy positions and entailing positions.
\end{abstract}
\section{Introduction}
Let us consider a simple piece-moving ruleset {\sc tripiece}.
There is a rectangular board with several pieces on it, including round, triangular, and square pieces.
Each player, on their turn, pick one piece and move it to up or to left. One piece can move to a place even other piece is already placed on there.
Further, on each place $(x,y)$ that satisfies $x+y=3$, something special happens to the triangular and square pieces.
\begin{itemize}
    \item [(a)]
    If a triangular piece is on $(x,y)(x+y = 3),$ instead of the usual vertical or horizontal move, the player can move it to $(x+1, y-1)$ or $(x-1, y+1)$ unless the place is outside of the board.
    \item [(b)]
    If a player moves a square piece to $(x, y)(x+y = 3),$ the player can choose ending the move or picking any piece and making one additional move.
\end{itemize}

The player who cannot move on their turn loses the game.

Figure \ref{fig:trpiece} is an example of a position of this ruleset.
Then who can win? Or is the best result of each player draw?

\begin{figure}
    \centering

    \begin{tabular}{|c||c|c|c|c|c|c|c|c|c|}\hline
         &\ 0~ &\ 1~ &\ 2~ &\ 3~ &\ 4~ &\ 5~ &\ 6~ &\ 7~ \\ \hline\hline
        0 &   &   &   &   &   &   &   &    \\ \hline
        
        1 &   &   & $\bullet$  &   &   &   &  $\blacktriangle$ &  \\ \hline 
        
        2 &   &   &   &   &   &   &    &    \\ \hline
        3 &   &   &   &  $\blacktriangle$ &   &    &   &    \\ \hline
        4 &   &   &   &   &   &   &   &    \\ \hline
         5 &   &   &   &   & $\bullet \blacksquare$  &   &   &    \\ \hline
         6 &   &   &   & $\blacksquare$  &   &   &   &   \\ \hline
         7 &   &   &   &   &   &   &   &    \\ \hline
                
    \end{tabular}
    
    \caption{Position and cordinate of {\sc tripiece}.}
    \label{fig:trpiece}
\end{figure}

The place where one piece can move is independent from each other. Thus, this game can be considered as a disjunctive sum of single-piece positions. However, since triangular pieces and square pieces have special moves, we cannot use standard Sprague-Grundy theory. For impartial games with such that one position can happen more than twice, like the move of triangular pieces, Sprague-Grundy theory is extended in \cite{Smi66}, \cite{FT75, FP73, FY86}.
For impartial games with carry-on moves like square pieces, Sprague-Grundy is also extended in \cite{LNS21}.
However, both has never been combined. In this study, we show partial construction of adding both games with loopy moves and games with carry-on moves.

\subsection{Early result: impartial games}

In combinatorial game theory, we say one of the both players Left (female) and the other Right (male).
For any position $G$, $\lopt{G}$ is a Left option of $G$ if Left can move from $G$ to $\lopt{G}$. Similarly, $\ropt{G}$ is a Right option of $G$ if Right can move from $G$ to $\ropt{G}$.

\begin{definition}
For any positions, 
    $\lopt{G}_1, \lopt{G}_2, \ldots, \lopt{G}_n, \allowbreak \ropt{G}_1, \ropt{G}_2, \ldots, \ropt{G}_m$, 
    $$\{\lopt{G}_1, \lopt{G}_2,\allowbreak \ldots, \lopt{G}_n \mid \ropt{G}_1, \ropt{G}_2, \ldots, \ropt{G}_m\}$$ is also a position. 
    In this form, the set of Left options $G^\mathcal{L} = \{\lopt{G}_1, \lopt{G}_2, \ldots, \lopt{G}_n \}$ and the set of Right options $G^\mathcal{R} = \{\ropt{G}_1, \ropt{G}_2, \ldots, \ropt{G}_m\}$ can be empty sets.
    
    We may write this position as $\{G^\mathcal{L} \mid G^\mathcal{R}\}$ 
\end{definition}

Every position $G$ belongs its outcome $o(G) \in \{\mathcal{L}, \mathcal{R}, \mathcal{N}, \mathcal{P}\}$. $o(G) = \mathcal{L}$ means that in $G$ Left has a winning strategy regardless of whether she is the next player or the previous player. Similarly, $o(G) = \mathcal{R}, o(G) = \mathcal{N}$, and $o(G) = \mathcal{P}$ mean that in $G$ Right, the next player and the previous player has a winning strategy, respectively.

\begin{definition}
For any positions $G$ and $H$, $G \cong H$, or 
$G$ and $H$ are isomorphic
 if $G^\mathcal{L} = H^\mathcal{L}$ and  $G^\mathcal{R} = H^\mathcal{R}$.
\end{definition}

\begin{definition}
    A ruleset $\Gamma$ is an impartial ruleset if every $G \in \Gamma$ satisfies $G^\mathcal{L} = G^\mathcal{R}$.

\end{definition}

In impartial games, the Left options are the same as the Right options, so we just say options for them.
When $G'$ is an option of $G$, we write $G \rightarrow G'$.

An impartial position $G$ belongs to exactly one of $\mathcal{N}$ and $\mathcal{P}$.
We can determine the outcome of $G$ recursively such that if $G$ has an option $G'$ which satisfies $o(G') = \mathcal{P}$, then $o(G) = \mathcal{N}$ and every $G'$ which is $G \rightarrow G'$ satisfies $o(G') = \mathcal{N}$, then $o(G) = \mathcal{P}$.

Let $\afimp$ be the set of all positions in impartial rulesets.
Let disjunctive sum of $G$ and $H$ be
$$
G+H \cong \{G^\mathcal{L} + H, G + H^\mathcal{L} \mid G^\mathcal{R} + H, G+ H^\mathcal{R}\}.
$$

That is, the disjunctive sum of positions $G$ and $H$ is a position such that each player chooses one component and makes a move to the component in their turn.

\begin{definition}
    For any positions $G, H \in \afimp$, $G=_{\afimp} H$ means for any $X \in \afimp$, $o(G+X)=o(H+X)$ holds.
\end{definition}

Let $\mathbb{N}_0$ be the set of all nonnegative integers.
    For any set of nonnegative integers $A$, we let
    $$\mex({A}) = \min (\mathbb{N}_0\setminus A).$$

\begin{definition}

For any position $G$ of impartial games, The Sprague-Grundy value of $G$ is 
    $$
\mathcal{G}(G) = \mex(\{ \mathcal{G}(G') \mid G \rightarrow G'\}).
    $$.
\end{definition}

Sprague-Grundy value is independently found by Sprague and Grundy (\cite{Spr35, Gru39}). By using Sprague-Grundy value, we can determine the outcome of positions as :
$$ o(G) = 
        \left \{ \begin{array}{ll}
            \mathcal{P} & \mathcal{G}(G)  = 0\\
            \mathcal{N} & \mathcal{G}(G)  \neq 0. 
        \end{array} \right .
        $$

Further, 
$$\mathcal{G}(G+H) = \mathcal{G}(G) \oplus \mathcal{G}(H)$$
holds. (Here, $\oplus$ is the operator of exclusive OR.)

Thus, if a position is disjunctive sum of positions, we need to know only Sprague-Grundy values of all components to determine the outcome of the whole position.

\begin{definition}
    Let 
    $$\nimber {0} \cong \{ \mid \}$$
    and 
    $$
    \nimber{n} \cong
    \{ \nimber{0}, \nimber{1}, \ldots, \nimber{(n-1)} \mid \nimber{0}, \nimber{1}, \ldots, \nimber{(n-1)} \}.
    $$
    We call $\nimber{0}, \nimber{1}, \ldots$ are {\em nimbers}.
    $\nimber{0}, \nimber{1}$ are abbreviated simply as $0$ and $\sglstar$ respectively.
    
\end{definition}

The outcome of the position depends on the Sprague-Grundy value of each component. Therefore, all positions whose Sprague-Grundy values are the same are equivalent when we use $=_{\afimp}$ as the equivalence relasion.

Since $\nimber{k}$ has Sprague-Grundy value $k$, for every $G\in \afimp$, $G =_{\afimp} \nimber{\mathcal{G}(G)}$ holds.

Therefore, the equivalence class $\afimp /{=_{\afimp}}$ is isomorphic to $\{ \nimber{n} \mid n \in \mathbb{N}_0\}$.
\begin{corollary}
    $$\nimber{k_1} + \nimber{k_2} =_{\afimp} \nimber{(k_1 \oplus k_2)}
    $$
    holds.
    
\end{corollary}
\begin{corollary}
\label{cor:normalimp} For
    $$
    G =_{\afimp} \nimber{k_1} + \cdots + \nimber{k_s}, 
    $$
     
    $$ o(G) = 
        \left \{ \begin{array}{ll}
            \mathcal{P} &(k_1 \oplus \cdots \oplus k_s  = 0) \\
            \mathcal{N} &(k_1 \oplus \cdots \oplus k_s  \neq 0)
        \end{array} \right .
        $$
        holds.
\end{corollary}

\subsection{Early result: impartial loopy games}
Impartial loopy game is impartial game which is not guaranteed it ends in finite moves. Sprague-Grundy theory is extended to loopy games (see \cite{Smi66} and \cite{FP73, FT75, FY86}.)

In impartial loopy games, the game may never end. Such situation is called Draw and the position is called a $\mathcal{D}$- position.

\begin{definition}
For any nonnegative integer $n$, $\mathcal{P}_n$ is the set of positions such as: 
    \begin{itemize}
        \item $G \in \mathcal{P}_0$ $\Longleftrightarrow$ $G$ is a terminal position.
        \item $G \in \mathcal{P}_{n}$ $\Longleftrightarrow$ For each $G'$,  which is an option of $G$, there is an option $G'' \in \mathcal{P}_k$ $(k<n)$.
    \end{itemize}
\end{definition}
\begin{definition}
\label{def_loopy_4.5}
For any position $G$ in impartial loopy ruleset, 
    \begin{itemize}
        \item $G$ is a $\mathcal{P}$-position if and only if there is a nonnegative integer $n$ such that $G \in \mathcal{P}_n$.
        \item $G$ is an $\mathcal{N}$-position if and only if $G$ has an option $G'$, which is a $\mathcal{P}$-position.
        \item $G$ is a $\mathcal{D}$-position if and only if $G$ is not a $\mathcal{P}$-position nor $\mathcal{N}$-position.
    \end{itemize}
\end{definition}

\begin{theorem}
\label{thm_loopy_4.6}
    For any position $G$ in impartial loopy ruleset, the following hold.
    \begin{itemize}
        \item [(a)] Every option of $G$ is an $\mathcal{N}$-position. $\Longleftrightarrow$~$G$ is a $\mathcal{P}$-position.
        \item [(b)] An option of $G$ is a $\mathcal{P}$-position. $\Longleftrightarrow$~$G$ is an $\mathcal{N}$-position.
        \item [(c)] No option of $G$ is a $\mathcal{P}$-position and an option of $G$ is a $\mathcal{D}$-position. $\Longleftrightarrow$~$G$ is a $\mathcal{D}$-position.
    \end{itemize}
\end{theorem}
\begin{theorem}
    For any position $G$ in impartial loopy ruleset, the following hold:
    \begin{itemize}
        \item [(a)] $G$ is a $\mathcal{P}$-position. $\Longleftrightarrow$ The previous player has a winning strategy.
        \item[(b)] $G$ is an $\mathcal{N}$-position. $\Longleftrightarrow$ The next player has a winning strategy.
        \item[(c)] $G$ is a $\mathcal{D}$-position. $\Longleftrightarrow$ Each player can make it draw.
    \end{itemize}
\end{theorem}

Let the set of all positions in impartial loopy ruleset $\limp$.
 Note that $\afimp \subset \limp$.

\begin{definition}

    We say $G =_{\limp} H$
    if for any $X \in \limp, o(G+X) = o(H+X)$.
    
\end{definition}

Note that even if $G =_{\afimp} H$, there might be a position $X \in \limp$ such that $o(G + X) \neq o(H + X)$ and $G \neq_{\limp} H.$
However, in fact, it was solved that such $X$ does not exist and $G =_{\afimp} H \Longrightarrow G =_{\limp} H$. Further, some positions do not equal to any $\nimber{n}$.  
\begin{definition}
    For any impartial loopy position $G$ and integer $n \geq 0$, let $\mathcal{G}_0(G)\in \mathbb{N}_0 \cup \{\infty\}$ be
$$
\mathcal{G}_0(G) = \left\{
\begin{array}{ll}
0 & \text{$G$ is a terminal position.} \\
\infty & \text{Otherwise.}
\end{array}
\right.
$$
Further, for any $n \geq 0$ let $m_n(G) = {\rm mex}(\{\mathcal{G}_n(G')\mid G \rightarrow G' \})$ and let
$$
\mathcal{G}_{n+1}(G) = \left\{
\begin{array}{ll}
m_n(G) & \text{If for every $G',$ which is an option of $G$ and satisfies $\mathcal{G}_n(G') > m_n(G)$,} \\ &\text{there is an option $G''$ of $G'$ such that $\mathcal{G}_n(G'') = m_n(G)$.} \\
\infty & \text{Otherwise.}
\end{array}
\right.
$$
Here, for any finite set $T \subset \mathbb{N}_0 \cup \{\infty\}$, ${\rm mex}(T) = \min ((\mathbb{N}_0\cup \{\infty\}) \setminus T)$.
\end{definition}

\begin{definition}
\label{def:loopygvalue}
    If $\mathcal{G}_n(G) = m$ for enough large $n$, we write $\mathcal{G}(G) = m$.
    Otherwise, let $\mathcal{G}(G) = \infty(\mathcal{A})$. Here, $\mathcal{A} = \{\mathcal{G}(G') \in \mathbb{N}_0\mid~G \rightarrow G'\}$.
\end{definition}
It is known that if we choose enough large $n,$ then $\mathcal{G}(G)$ is determined and will not change even if we choose larger $n$. 

For any $G \in \afimp$, from the Definition \ref{def:loopygvalue},  $\mathcal{G}(G)$ is the same as Sprague-Grundy value of $G$ defined in the last section. Therefore, we use the same symbol $\mathcal{G}(G)$.

\begin{theorem}$\mathcal{G}(G)$ satisfies the followings: 
    \begin{itemize}
        \item [(a)]$\mathcal{G}(G) = 0\Longleftrightarrow G$ is a $\mathcal{P}$-position.
        \item [(b)]$\mathcal{G}(G)$ is an integer $m>0$ $\Longrightarrow G$ is an $\mathcal{N}$-position.
        \item [(c)]$\mathcal{G}(G) = \infty(\mathcal{A})$ and $0 \in \mathcal{A} \Longrightarrow G$ is an $\mathcal{N}$-position.
        \item [(d)]$\mathcal{G}(G) = \infty(\mathcal{A})$ and $0 \not \in \mathcal{A}\Longleftrightarrow G$ is a $\mathcal{D}$-position.
    \end{itemize}
\end{theorem}
When $\mathcal{G}(G) = \infty(A)$, we write $G =_{\limp} \infty(A)$.

The following theorems holds.
\begin{theorem}
\label{thm_loopy_4.14}
    If $\mathcal{G}(G) = m$, then $G =_{\limp} *m$.
\end{theorem}

\begin{theorem}
$$
\nimber{n_1} + \nimber{n_2} =_{\limp} \nimber{(n_1 \oplus n_2)}
$$    
$$
\nimber{n} + \infty(A) =_{\limp} \infty(n \oplus A)
$$
$$
\infty(A) + \infty(B) =_{\limp} \infty(\emptyset).
$$
Here, for any integer $n$ and set $A$, $n \oplus A = \{n \oplus a \mid a \in A\}$.
\end{theorem}

In summary, a position in impartial loopy game equals to $\nimber{n}$ or $\infty(A)$ by equivalence relation $=_{\limp}$ and the equivalence class $\limp /{=_{\limp}}$ is isomorphic to $\{ \nimber{n} \mid n \in \mathbb{N}_0 \} \cup \{\infty(A) \mid A \subsetneq \mathbb{N}_0\}$.

\subsection{Early result: impartial entailing games}

In some rulesets, there are carry-on moves. That is, a player can move twice or more under some special circumstances. 
The theory produced to handle such rulesets is the theory of {\em entailing games}. In this section, we overview the theory of impartial loop-free entailing games introduced in \cite{LNS21}. See also \cite{LNS23} for the ruleset analyzed by using the theory of entailing games and \cite{LNS24} for partisan entailing games.

To describe entailing games, positive infinity $\infty$ and negative infinity  $\minfty$ are introduced \footnote{Note that following the early results, we use the symbol $\infty$ for both entailing games and loopy games. $\infty(A)$, with a set $A$, is a position in loopy games and only $\infty$ is  a position in entailing games.}. These positions are immediately win for Left and Right. That is, once a component of disjunctive sun becomes $\infty$ (resp. $\minfty$), then Left (resp. Right) wins whatever the other components are.
Consider a situation where Left moves from an impartial  game position $G$ to the position $G'$ and also has the right to make the next move. To describe this situation, we say that Left has $\{ \infty \mid G'\}$ as an option. If left chooses this move, Right will always lose except by choosing $G'$, so Right has no choice but to immediately move to $G'$. As a result, Left makes her turn at $G'$, and we can describe a sequence of moves from $G$ to $G'$ and then back to one's own turn.

Of course, it is important to note that the introduction of such symbols can represent a variety of situations, not just a simple carry-on move.
For example, if there is a left option
$\{ \infty \mid G', G''\}$, then Left can force Right to choose $G'$ or $G''$, but its not exactly the same as carry-on move since this includes Right's choice.
Let the set of all  entailing games $\nplay^\infty$.
The positions of entailing games are defined recursively as well as standard partisan positions. This recursively definition starts from $\infty, \minfty$ while standard theory starts from $\{ \mid \}$. Since $\{ \minfty \mid \infty\}$ plays the same role as $\{ \mid \},$ we can consider standard partisan games in this framework.

\begin{axiom}
\label{axm:base}
The positive infinity $\infty$ and the negative infinity $\overline{\infty}$ satisfy the followings: 
\begin{enumerate}
    \item $\infty \in \mathcal{L}.$
    \item $\overline{\infty} \in \mathcal{R}.$
    \item For any $X \in \nplay^\infty \setminus \{\overline{\infty}\}$, $\infty + X = \infty.$ 
    \item For any $X \in \nplay^\infty \setminus \{\infty\}$, $\overline{\infty} + X = \overline{\infty}.$ 
    \item $\infty + \overline{\infty}$ is not defined.
\end{enumerate}

\end{axiom}

Next, we introduce impartial entailing games (or affine impartial games).

\begin{definition}[Check]
    Consider $G \in \nplay^\infty$. If $\infty \in G^\mathcal{L}$( resp. $\overline{\infty} \in G^\mathcal{R}$), then we call this a {\em Left-check (resp. Right-check)}. If $G$ is a Left-check or a Right-check, then we say that $G$ is a {\em check}. $G^{\overrightarrow{L}}(G^{\overleftarrow{R}})$ is a Left option (resp. Right option) of $G$ which is a Left-check (resp. Right-check).
\end{definition}

\begin{definition}
    Assume that $G \in \nplay^\infty$. If $G \neq \infty$ (resp. $G \neq \overline{\infty}$) and $G$ is not Left-check (resp. Right-check), then $G$ is Left-quiet (resp. Right-quiet). If $G$ is both Left-quiet and Right-quiet, then $G$ is quiet.
\end{definition}

\begin{definition}
    For any position $G \in \nplay^\infty$, the conjugate $\conj{G}$ is 

$$
\conj{G} = \left\{
\begin{array}{ll}
\overline{\infty} & (G = \infty) \\
\infty & (G = \minfty) \\
\{ \conj{G^\mathcal{R}} \mid \conj{G^ \mathcal{L}} \} & (\text{Otherwise.})
\end{array}
\right.
$$
Here, $\conj{G^ \mathcal{L}}$ is the set of all conjugates of Left options of $G$. Similarly, $\conj{G^ \mathcal{R}}$ is the set of all conjugates of Right options of $G$.
    
\end{definition}

\begin{definition}[Symmetric position]
    Assume that $G \in \nplay^\infty$. $G$ is symmetric if $G \not \in \{ \infty , \overline{\infty}\}$ and $G^\mathcal{R} = \conj{G^\mathcal{L}}$.
\end{definition}
\begin{definition}[Impartial entailing game]
    If $G \in \nplay^\infty$ is symmetric and every quiet follower of $G$ is also symmetric, then $G$ is a position of impartial entailing games. 
    Let the set of all positions of impartial entailing games be $\afimp^\infty \subset \nplay^\infty$.
    
\end{definition} 
\begin{definition}
    Let
    $$
    \moon \cong \{ \infty \mid \minfty\}.
    $$
    
\end{definition}

In impartial entailing games, there is a position such that the outcome is $\mathcal{N}$ whatever positions are added to it. The simplest example is $\moon\cong \{ \infty \mid \minfty\}$, but many positions like $\{\nimber{2}, \{ \infty \mid \nimber{2}\} \mid \nimber{2}, \{\nimber{2} \mid \minfty\}\}$ satisfies the property. 

\begin{definition}
For any positions $G, H \in \afimp^\infty$, if for any $X \in \afimp^\infty $, $o(G+X)= o(H+X)$ holds, then $G =_{\afimp^\infty} H$.
\end{definition}
Similar to when we define $=_{\limp}$, we should consider that there might be positions such that $G=_{\afimp} H$ and $G \neq_{\afimp^\infty} H$. However, it is proved that, there are no such positions and , if $G=_{\afimp} H$ then $G =_{\afimp^\infty} H$.
In addition, some $G  \in \afimp^\infty \setminus \afimp$ satisfies $G =_{\afimp^\infty} \nimber{n}$. Thus we need to extend the Sprague-Grundy theory as well as impartial loopy games.

\begin{definition}
Let $\nim$ be the set of positions $G$ such that $G \in \afimp^\infty$ and $G =_{\afimp^\infty} \nimber{n}$ for an integer $n$.
\end{definition}
\begin{definition}
\label{def10}
    Assume that $G \in \afimp^\infty$. The set of {\em $G$-immediate nimbers} is $S_G = G^\mathcal{L} \cap \nim$.
\end{definition}

Note that by symmetry $S_G = G^{\mathcal{R}} \cap \nim$, and $S_{\moon} = \emptyset$.

\begin{definition}
    Assume that $G \in \afimp^\infty$. The set of {\em $G$-protected nimbers} $P_G$ is defined as follows: 
    \begin{enumerate}
    \item If $\infty \in G^\mathcal{L}$, then $P_G = \nim$.
    \item Otherwise, $P_G = \{*n \mid G^{\overrightarrow{L}} + *n \in \mathcal{L}, G^{\overrightarrow{L}} \in G^\mathcal{L}\}$. 
    \end{enumerate}
\end{definition}

\begin{theorem}
    \label{thm19}
    $\moon$ is an absorbing element in $\afimp^\infty$, that is, for any $Y \in \afimp^\infty$, $\moon + Y =_{\afimp^\infty} \moon$ holds.
\end{theorem}
\begin{theorem}
\label{thm:loonyismoon}
    Assume that $G \in \afimp^\infty$. If for any $\nimber{k}$, $o(\nimber{k} + G ) = \mathcal{N}$, then $G =_{\afimp^\infty} \moon$.
\end{theorem}

\begin{theorem}
\label{thm20}
    Assume that $G \in \afimp^\infty$. The following holds.
    \begin{itemize}
        \item If $S_G \cup P_G = \nim$, then $G =_{\afimp^\infty} \moon$, and 
        \item If $S_G \cup P_G \neq \nim$, then $G =_{\afimp^\infty} *({\rm mex}(\mathcal{G}(S_G \cup P_G)))$.
    \end{itemize}
\end{theorem}

In summary, a position in impartial entailing game equals to $\nimber{n}$ or $\moon$ by equivalence relation $=_{\afimp^\infty}$ and the equivalence class $\afimp^\infty /{=_{\afimp^\infty}}$ is isomorphic to $\{ \nimber{n} \mid n \in \mathbb{N}_0 \} \cup \{\moon \}$.

\section{Main results}

In this study, we consider {\em impartial entailing loopy games}.
In standard impartial entailing games, every position which is not equal to nimber is equal to $\moon$. However, if we add loopy games to them, the equivalence changes.

Consider an impartial loopy position $G \cong \{\infty(\emptyset), \sglstar \mid \infty(\emptyset), \sglstar\}$ and two imparital entailing position $H \cong \{\infty \mid \minfty \}$ and $J \cong \{ *2, \{\infty \mid *2\} \mid *2, \{*2 \mid \minfty\} \}$. In impartial entaling games, $H$ and $J$  equal to  $\moon$, however, if we add them to $G$, the outcomes are different.
$o(G+H) =  \mathcal{N}$ still holds. However, in $G+J$, the next player has no winning move and $o(G + J) = \mathcal{D}$, so we cannot say $H = J$ since the outcomes of $G + H$ and $G+J$ are different.
Therefore, we need to classify moons into more than one equivalence classes.

In this study, we consider $\ims^\infty$, which is a subset of the set of all impartial entailing positions, the set of all positions in impartial loopy positions $\limp$, and the sum of positions in  $\ims^\infty$ and $\limp$ together. We will show their algebraic structure and how we can determine the outcome of such positions.

\begin{definition}
\label{def:specialmoon}
    Let $n$ be an integer. The {\em special moon} of $n$ is 
    $$\spmoon(n)_A \cong \{ \{ \infty \mid \nimber{n} \}, A \mid A, \{ \nimber{n} \mid \minfty \} \}.$$
    Here, $A$ is a set of nimbers and special moons, which must include $\nimber{n}$.
\end{definition}

Note that this definition of special moon includes special moon itself. However, in this study, we assume that entailing positions are short, that is, only finite moves are required to end. Thus this definition is recursively definition and circular references never occur.

\begin{lemma}
\label{lem:spmoonismoon}
    $\spmoon(n)_A =_{\afimp^\infty} \moon$.
\end{lemma}

\begin{proof}
    For any $\nimber{k}$, we show $\spmoon(n)_A + \nimber{k} \in \mathcal{N}$.
    When $k = n$, Left can make $\spmoon(n)_A + \nimber{k}$ to $\nimber{k} + \nimber{k} \in \mathcal{P}$ by one move. When $k \neq n$, Left can make $\spmoon(n)_A + \nimber{k}$ to $\{\infty \mid \nimber{n}\} +\nimber{k}$ by one move, and force Right to make it $\nimber{n} + \nimber{k} $. Since $k \neq n$, this position belongs to $\mathcal{N}$. Therefore, $\spmoon(n)_A + \nimber{k} \in \mathcal{N}$ and from Theorem \ref{thm:loonyismoon}, $\spmoon(n)_A =_{\afimp^\infty} \moon$.
\end{proof}

\begin{definition}
\label{def:ims}
Let $\ims^\infty \subset \afimp^\infty$ be the set of all positions $G$ such that $G =_{\afimp^\infty} \nimber{k}$ for an integer $k$, $G \cong \moon$, and $G \cong \ \spmoon(n)_A$ for a nonnegative integer $n$ and a set $A$, which includes $\nimber{n}$.   
\end{definition}

\begin{definition}
\label{def:elimp}
Let $\elimp$ be the set of all positions $G_1 + \cdots + G_n + H_1 + \cdots + H_m$, where $G_1,  \ldots, G_n \in \ims^\infty$ and $H_1,  \ldots, H_m \in \limp$.
\end{definition}

\begin{definition}
\label{def:moonandlimp}
    For any $X \in \elimp$, $o(\infty  + X) = \mathcal{L}$ and $o(\minfty + X) = \mathcal{R}$.
\end{definition}
\begin{corollary}
\label{cor:moonandlimp}
    For any $X \in \elimp$, $o(\moon  + X) = \mathcal{N}$. 
\end{corollary}

\begin{theorem}
\label{thm:elimpndp}
    If $G \in \elimp$, then $G \in \mathcal{N} \cup \mathcal{D} \cup \mathcal{P}$.
\end{theorem}

\begin{proof}

    If Left can win $G \in \elimp$ by playing first, Right can also win by using the same strategy. Similarly, if Left can make draw by playing first, Right can make draw by playing first as well.
    
\end{proof}

\begin{definition}
\label{def:eqeq}
    For positions $G$ and $H \in \elimp$, if for any $X \in \elimp$, $o(G+X) = o(H+X)$ then we say $G \eqeq H$.
\end{definition}

\begin{lemma}
\label{lem:eqeq}
    For positions $G, H \in \ims^\infty$,  if $G \eqeq H$ then $G =_{\afimp^\infty} H$. Similarly, for positions $G, H \in \limp$,  if $G \eqeq H$ then $G =_{\limp} H$.
\end{lemma}

\begin{proof}
    Assume that $G, H \in \ims^\infty$ and $G \eqeq H$. Then for any $X \in \ims^\infty \subset \elimp$, $o(G+X) = o(H+X)$. Thus, $G =_{\ims^\infty} H$. Therefore, $G =_{\afimp^\infty} H$.
    Also, assume that $G, H \in \limp$ and $G \eqeq H$. Then for any $X \in \limp \subset \elimp$, $o(G+X) = o(H+X)$. Thus, $G =_{\limp} H$.
\end{proof}

\begin{theorem}
\label{thm:elimpp}
    Assume that $G \in \elimp$. Then $G \in \mathcal{P}$ if and only if, $G \eqeq 0$.
\end{theorem}

\begin{proof}
    If $G \eqeq 0$, then from $o(G) = o(0) = \mathcal{P}$, we have $G\in \mathcal{P}$. Conversely, assume that $G \in \mathcal{P}$. For any $X + G$, if one player has a winning strategy in $X$, the player can use the strategy in $X$ and in $G$, they only use the winning strategy for the previous player when the opponent moved in $G$. Next, if it is the best to make draw in $X$ for each player, in $G+X$, it is the best to use the same strategy in $X$ and use the winning strategy for the previous player in $G$. Therefore, the outcomes of $X$ and $G+X$ are exactly the same, and we have $G \eqeq 0$.
\end{proof}

\begin{corollary}
\label{cor:eqeq}
    The relation     $\eqeq$ is an equivalence relation. In addition, the followings holds:  
    \begin{itemize}
        \item $G \eqeq H \Longrightarrow G + J \eqeq H + J$.
        \item If $J\in \elimp$ has an inverse $J'$, then $G + J \eqeq H + J \Longrightarrow G \eqeq H$.
    \end{itemize}
    
\end{corollary}

\begin{theorem}
\label{thm:geqh}
    For any $G, H \in \elimp$, if $o(G + H) = \mathcal{P}$ and $o(H + H) = \mathcal{P}$, then $G \eqeq H$.
\end{theorem}

\begin{proof}
    Since $o(G + H) = \mathcal{P}$, $G + H \eqeq 0$, and since $o(H+H) = \mathcal{P}$, $H+H \eqeq 0$. Thus, $G \eqeq G + (H+H) \eqeq (G + H) + H \eqeq H$. Therefore, $G \eqeq H$.
\end{proof}

Next, we consider whether equivalence holds even if the set under consideration is expanded to $\elimp$

\begin{lemma}
\label{lem:eqnimber}
Assume that for $G, H \in \elimp$, $G =_{\afimp^\infty} \nimber{n}$ or $G =_{\limp} \nimber{n}$ holds, and $H =_{\afimp^\infty} \nimber{n}$ or $H =_{\limp} \nimber{n}$ holds. Then, $G \eqeq \nimber{n} \eqeq H$.
    
\end{lemma}

\begin{proof}
    When $G =_{\afimp^\infty} \nimber{n}$, from the theory of impartial entailing games, $o(G + \nimber{n}) = \mathcal{P}$. Then, from Theorem \ref{thm:geqh}, $G \eqeq \nimber{n}$ holds.
    When $G =_{\limp} \nimber{n}$, from the theory of impartial loopy games, $o(G + \nimber{n}) = \mathcal{P}$. Then, from Theore \ref{thm:geqh}, $G \eqeq \nimber{n}$ holds.
    Similarly, $H \eqeq \nimber{n}$. Therefore, we have
    $$
G \eqeq \nimber{n} \eqeq H.
    $$
\end{proof}

By using this theorem, a position which equal to $\nimber{n}$ in impartial entailing games or impartial loopy games, can be treated as $\nimber{n}$ without worrying whether it belongs to $\ims^\infty$ or $\limp$.

Next, we consider special moon. Then, differ from nimbers, even if $G, H \in \ims^\infty$ satisfy $G =_{\afimp^\infty} H$, sometimes $G \not \eqeq H.$ 
For example,  $\spmoon(1)_A =_{\afimp^\infty} \moon$ and $\spmoon(1)_A + \infty(\emptyset) \in \mathcal{D}$ while $\moon + \infty(\emptyset) \in \mathcal{N}$.
    
However, we have the following lemma.
\begin{lemma}
\label{lem:spnd}
    For any $X \in \elimp$, $\spmoon(n)_A + X \in \mathcal{N} \cup \mathcal{D}$ holds.
\end{lemma}

\begin{proof}
    If $X \eqeq \moon$, then obviously $\spmoon(n)_A + X \in \mathcal{N}$. We consider the other cases.
    Then, for $\spmoon(n)_A + X$, Left has a move to $\nimber{n} + X$ and a move to $\{ \infty \mid \nimber{n}\} + X$. For the later move, Right has no choice other than to make it $\nimber{n} + X$. Then, if $\spmoon(n)_A + X$ is a $\mathcal{P}$-position, a option $\nimber{n} + X$ must be an $\mathcal{N}$-position but the next player can force the other player to make the position $\nimber{n} + X$, which causes a contradiction. Therefore, $\spmoon(n)_A + X$ is not a $\mathcal{P}$-position and from Theorem \ref{thm:elimpndp}, it is an $\mathcal{N}$-position or a $\mathcal{D}$-position.
\end{proof}

\begin{lemma}
\label{lem:loopyeq}
Assume that $G =_{\limp} \infty(A)$ and $ H =_{\limp} \infty(B)$ for any set of nonnegative integers $A$ and $B.$ Then,   $G + H \eqeq \infty(\emptyset)$. 
\end{lemma}

\begin{proof}
When $X \eqeq \moon$, $o(G + H + X) = o(\infty(\emptyset) + X) = \mathcal{N}$ holds.

    For any $X \in \elimp (X \not \eqeq \moon)$, we prove $o(G + H + X) = \mathcal{D}$ by induction on the numbers of special moons as components of disjunctive sum in $X$.
    If there is no $\spmoon_A(n)$ in $X$, from the theory of impartial loopy games, $G + H =_{\limp} \infty(\emptyset)$. Thus, $o(G + H + X) = \mathcal{D}$.

    Assume that $o(G + H + X) = \mathcal{D}$ if there are $k$ special moons in  $X$ as components of disjunctive sum. 
    Let $X = X' + \spmoon(n_1)_{A_1} + \spmoon(n_2)_{A_2} + \cdots + \spmoon(n_{k+1})_{A_{k+1}}$ when there are $k+1$ special moons in $X$. Here, in $X'$, there is no special moon as a component of disjunctive sum.
    Since $X$ has special moons as components, from Lemma \ref{lem:spnd},  this position is an $\mathcal{N}$-position or a $\mathcal{D}$-position. Assume that $X$ is an $\mathcal{N}$-position and consider the winning strategy for Left when she is the first player. If Left can win by changing $X'$ to $X''$, since in $X'$ there is no special moon, this move is not check. Therefore, $X'' \in \elimp$ and $(G + H + X'' + \spmoon(n_1)_{A_1} + \cdots + \spmoon(n_{k+1})_{A_{k+1}}) \in \mathcal{N} \cup \mathcal{D} \cup \mathcal{P}$. From Lemma \ref{lem:spnd}, $o(G + H + X'' + \spmoon(n_1)_{A_1} + \cdots + \spmoon(n_{k+1})_{A_{k+1}}) \neq \mathcal{P}$, which contradicts the assumption that Left can win by this move.
    
    Next, consider the case that Left changes $\spmoon(n_i)_{A_i}$ to a $\nimber{n}\in A_i$. Then there are $k$ special moons as components of the disjunctive sum, so from the induction hypothesis, this position is a $\mathcal{D}$-position. Therefore, this is not a winning move for Left. Also, if Left changes $\spmoon(n_i)_{A_i}$ to $\spmoon(n')_{A'}\in A_i$, then the whole position has a special moon as a component, so the whole position is an $\mathcal{N}$-position or a $\mathcal{D}$-position. Thus, this is not a winning move for Left.

    Next, consider the case that Left changes $\spmoon(n_i)_{A_i}$ to $\{\infty \mid \nimber{ n_i}\}$. Immediately after Left's move, Right must  change it to $\nimber{n_i}$ and there are $k$ special moons in the obtained position. Then there are $k$ special moons in the whole position so it is a $\mathcal{D}$-position and Left does not win.

    Finally, consider the case that Left move to $G$ or $H$. Then, there are still special moons, so the whole position is an $\mathcal{N}$-position or a $\mathcal{D}$-position, which means the Left's move is not a winning move.

    Therefore, $G + H + X$ is not a $\mathcal{P}$-position neither an $\mathcal{N}$-position, which means it is a $\mathcal{D}$-position.
    Then we can conclude for any $X \not \eqeq \moon$, $o(G+H+X) = \mathcal{D}$.
    Similarly, we can prove that $o(\infty(\emptyset) +X) = \mathcal{D}$ by induction and we have $G+H \eqeq \infty(\emptyset)$.
\end{proof}

\begin{lemma}
\label{lem:lospnd}
    Assume that $G =_{\limp} \infty(B)$ for a set of nonnegative integers $B$. Let $t$ and $k$ be nonnegative integers. Then, 
    the outcome of $\nimber{k} + \spmoon(n_1)_{A_1} + \cdots + \spmoon(n_t)_{A_t} + G$ is $\mathcal{N}$ when $(k \oplus n_1 \oplus \cdots \oplus n_t) \in B$ and otherwise $\mathcal{D}$.
\end{lemma}

\begin{proof}
    We prove it by induction on $t$. When $t = 0$, from the theory of impartial loopy games, it is an $\mathcal{N}$-position when $k \in B$ and otherwise it is a $\mathcal{D}$-position. Consider the case $t > 0$. Then, there is a component $\spmoon(n_1)_{A_1}$ of disjunctive sum, so from Lemma \ref{lem:spnd}, it is an $\mathcal{N}$-position or a $\mathcal{D}$-position.
    
    If $(k \oplus n_1 \oplus \cdots \oplus n_t) \in B$, then Left moves $\spmoon(n_t)_{A_t}$ to $\{ \infty \mid \nimber{n_t}\}$ when she is the next player. Then Right has to move there to $\nimber{n_t}$ immediately after the Left's move. Then the whole position changes to $\nimber{(k \oplus n_t)} + \spmoon(n_1)_{A_1} + \cdots + \spmoon(n_{t-1})_{A_{t-1}} + G$, and since $((k \oplus n_t) \oplus n_1 \oplus \cdots \oplus n_{t-1}) \in B$, from the induction hypothesis, this is an $\mathcal{N}$-position. Therefore, from the original position, Left can win when she is the first player. It means Right can also win when he is the first player.
    
    Next, we consider the case $(k \oplus n_1 \oplus \cdots \oplus n_t) \not \in B$.
    For each induction step on $t$, we also use induction on $k$.
    If Left changes $\nimber{k}$ to $\nimber{k'}(k'<k)$ when she is the first player,  then from the induction of $k$, the whole position is a $\mathcal{D}$-position.
    Next, if Left changes $\spmoon(n_i)_{A_i}$ to $\nimber{n}\in A_i$, then from the induction hypothesis on $t$,  the whole position is an $\mathcal{N}$-position or a $\mathcal{D}$-position. If Left moves $\spmoon(n_i)_{A_i}$ to $\spmoon(n')_{A'} \in A_i$, then the whole positions is an $\mathcal{N}$-position or a $\mathcal{D}$-position because it has a special moon as a component of disjunctive sum.
    Also, if Left changes $\spmoon(n_i)_{A_i}$ to $\{ \infty \mid \nimber{n_i}\}$, then Right moves it to $\nimber{n_i}$ immediately after Left's move. Then the position is a  $\mathcal{D}$-position from the induction hypothesis.
    
    Finally, if Left move to $G$, then still $\spmoon(n_1)_{A_1}$ is remaining as a component of disjunctive sum and the whole position is an $\mathcal{N}$-position or a $\mathcal{D}$-position. Therefore, Left cannot win but make draw when she is the first player, which means the original position is a $\mathcal{D}$-position.
    
\end{proof}
\begin{lemma}
    Assume that positions $G$ and $H$ satisfies $G =_{\limp} H =_{\limp} \infty(B)$ for a set of nonnegative integers $B$. Then, $G \eqeq H$.
\end{lemma}
\begin{proof}
   We show that for any position $X \in \elimp$, $o(G + X) = o(H + X)$. First, if a position $J =_{\limp} \infty(C)$ is a component of $X$, then from Lemma \ref{lem:loopyeq}, $o(G+X) = o(H+X)$. Also, for the case of $X = \nimber{k} + \spmoon(n_1)_{A_1} + \spmoon(n_2)_{A_2} + \cdots + \spmoon(n_t)_{A_t}$, from Lemma \ref{lem:lospnd}, $o(G+X) = o(H+X)$ holds.
\end{proof}

Therefore, the following corollary holds.

\begin{corollary}
    For any $G$ and $H \in \limp, $ if $G =_{\limp} H$ then $G \eqeq H$.
\end{corollary}
\begin{corollary}
The followings holds: 
$$
\nimber{n_1} + \nimber{n_2} \eqeq
\nimber{(n_1 \oplus n_2)}
$$    
$$
\nimber{n} + \infty(A) \eqeq \infty(n \oplus A)
$$
$$
\infty(A) + \infty(B) \eqeq
\infty(\emptyset)
$$
Here, for any integer $n$ and set $A$, $n \oplus A = \{n \oplus a \mid a \in A\}$.

\end{corollary}

The results so far can be summarized as follows.

\begin{theorem}
\label{thm:outcomeelimp}
    Let a position $G \in \elimp$ be
    \begin{eqnarray}G &\cong& b \cdot \moon + \nimber {k_1}  + \cdots + \nimber{k_s} + \spmoon(n_1)_{A_1}  + \cdots + \spmoon(n_t)_{A_t} \nonumber\\  &&+ \infty(B_1) + \cdots + \infty(B_u) \nonumber
    \end{eqnarray}
    . Then the outcome of $G$ is determined as follows: 
    \begin{itemize}
        \item [(i)] If $b \geq 1$, then $o(G) = \mathcal{N}$.
        \item [(ii)] If $b = 0$ and $u \geq 2$, then $o(G) = \mathcal{D}$.
        \item[(iii)] If $b = 0$ and $u = 1, $  then 
        $$ o(G) = 
        \left \{ \begin{array}{ll}
            \mathcal{N} & (k_1 \oplus \cdots k_s \oplus n_1 \oplus \cdots \oplus n_t \in B_1 )\\
            \mathcal{D} & (k_1 \oplus \cdots k_s \oplus n_1 \oplus \cdots \oplus n_t \not \in B_1 ) .
        \end{array} \right .
        $$
        
        \item [(iv)] If $b = 0, u = 0, $ and $t \geq 1$, then $o(G) = \mathcal{N}$.
        \item [(v)] If $b = 0, u = 0,$ and $t = 0$, then 
        $$ o(G) = 
        \left \{ \begin{array}{ll}
            \mathcal{P} & (k_1 \oplus \cdots \oplus k_s  = 0)\\
            \mathcal{N} & (k_1 \oplus \cdots \oplus k_s  \neq 0 ).
        \end{array} \right .
        $$
    \end{itemize}

\end{theorem}

\begin{proof}
    This follows from Corollary \ref{cor:moonandlimp}, Lemmas \ref{lem:loopyeq}, \ref{lem:lospnd}, \ref{lem:spmoonismoon}, Theorem \ref{thm19}, and Corollary \ref{cor:normalimp}.
\end{proof}

\begin{corollary}
Assume that $G \cong \spmoon(n)_{A_1}, H \cong \spmoon(n)_{A_2}$ for any nonnegative integer $n$ and sets $A_1, A_2$ which have $n$ as elements. Then, $G \eqeq H$.
\end{corollary}
\begin{proof}
    In Theorem \ref{thm:outcomeelimp}, the outcome does not depends on the form of $A_i$. Thus, $G \eqeq H$.
\end{proof}

From now on, $\spmoon(n)_A$ will be written simply as $\spmoon(n)$.

\begin{theorem}
    The followings hold:
    $$\spmoon(n) + \nimber{k} \eqeq \spmoon(n \oplus k)$$
    $$\spmoon(n_1) + \spmoon(n_2) \eqeq \spmoon(n_1 \oplus n_2)$$
    $$\spmoon(n) + \infty(B) \eqeq \infty(n \oplus B).$$
    Here, for any integer $n$ and set $B$,  $n \oplus B = \{n \oplus b \mid b \in B\}$.
\end{theorem}

\begin{proof}
For any $X\cong b \cdot \moon + \nimber {k_1}  + \cdots + \nimber{k_s} + \spmoon(n_1)  + \cdots + \spmoon(n_t) + \infty(B_1) + \cdots + \infty(B_u)$, from Theorem \ref{thm:outcomeelimp}, 
if $b \geq 1$ then $o(\spmoon(n) + \nimber{k} + X) = \mathcal{N}$, if $b = 0$ and $u \geq 2$ then $o(\spmoon(n) + \nimber{k} + X) = \mathcal{D}$, if $b = 0$ and $ u = 1$ then 
$$o(\spmoon(n) + \nimber{k} + X) = \left \{ \begin{array}{ll}
            \mathcal{N} & (
            n \oplus k \oplus k_1 \oplus \cdots k_s  \oplus n_1 \oplus \cdots \oplus n_t \in B_1 )\\
            \mathcal{D} & (n \oplus k \oplus k_1 \oplus \cdots k_s \oplus n_1 \oplus \cdots \oplus n_t \not \in B_1 ). 
        \end{array} \right .$$,
        if $b = 0$ and $u = 0$ then $ o(\spmoon(n) + \nimber{k} + X) =\mathcal{N}$. 
        Similarly, we can determine $o(\spmoon(n \oplus k) + X)$ by using Theorem 
        \ref{thm:outcomeelimp} and it is exactly the same as $o(\spmoon(n) + \nimber{k} + X)$.

        $\spmoon(n_1) + \spmoon(n_2) \eqeq \spmoon(n_1 \oplus n_2)$ and $,\spmoon(n) + \infty(B) \eqeq \infty(n \oplus B)$ can be proved in similar ways.
\end{proof}

\begin{table}[htb]
    \centering
    \begin{tabular}{c|cccc}
         & $\moon$ & $\nimber{k_1}$ & $\spmoon(n_1)$ & $\infty(B_1)$  \\ \hline
        $\moon$ & $\moon$ & $\moon$ & $\moon$ & $\moon$ \\
        $\nimber{k_2}$ & $\moon$ & $\nimber{(k_1 \oplus k_2)}$ & $\spmoon(n_1 \oplus k_2)$ & $\infty(B_1 \oplus k_2)$ \\ 
        $\spmoon(n_2)$ & $\moon$ & $\spmoon(k_1 \oplus n_2)$ & $\spmoon{(n_1 \oplus n_2)}$ & $\infty(B_1 \oplus n_2)$ \\
        $\infty(B_2)$ & $\moon$ & $\infty(k_1 \oplus B_2)$ & $\infty(n_1 \oplus B_2)$ & $\infty(\emptyset)$
    \end{tabular}
    \caption{Sum of positions in $\elimp$.}
    \label{tab:elimpsummary}
\end{table}

\begin{corollary}
Let $k$ and $n$  be nonnegative integers and $B$ be a set of nonnegative integers. 
    Any position $G \in \elimp$ satisfies exactly one of $G \eqeq \nimber{k}, G \eqeq \moon, G \eqeq \spmoon(n),$ and $ G \eqeq \infty(B)$. That is, the equivalent class $\elimp /{\eqeq}$ is isomorphic to $\{ \nimber{k} \mid k \in \mathbb{N}_0 \} \cup \{\moon \} \cup \{ \spmoon(n) \mid n \in \mathbb{N}_0\} \cup\{\infty(B) \mid B \subsetneq \mathbb{N}_0\} $.
\end{corollary}

We summary the sum of positions in $\elimp$ in Table \ref{tab:elimpsummary}.

Now the algebraic structure of $\elimp$ has been completely revealed. On the other hand, there are some impartial positions which is not included in $\elimp$. Thus, our next extension should be include all impartial entailing positions. 

However, they are more complex.
For example, consider two positions
$$
G \cong \{ \nimber{2}, \nimber{3}, \{ \infty \mid \nimber{2}, \nimber{3} \} \mid \nimber{2}, \nimber{3} , \{\nimber{2}, \nimber{3} \mid \minfty\} \}
$$
$$
H \cong \{ \nimber{2}, \nimber{3}, \{ \infty \mid \nimber{2}\}, \{ \infty \mid \nimber{3} \} \mid \nimber{2}, \nimber{3} , \{\nimber{2} \mid \minfty \}, \{ \nimber{3} \mid \minfty\} \}.
$$
 These positions equal to $\moon$ in the framework of impartial entailing games and have very similar form. However, when we consider the
necessary and sufficient condition to be $\mathcal{N}$ positions when we add the positions and impartial loopy game $\infty(B), $

$$
o(G + \infty(B)) = \mathcal{N} \Longleftrightarrow \{\nimber{2}, \nimber{3}\} \subset B 
$$
and 
$$
o(H + \infty(B)) = \mathcal{N} \Longleftrightarrow \nimber{2} \in B \text{ or } \nimber{3} \in B,
$$
 which means $G$ and $H$ are not equivalent.

Thus, just by examining a slightly extended range of positions  from the special moon, we can assume that a more complex structure than that discussed here is spreading.

Let us turn back to Figure \ref{fig:trpiece}. The round piece has the same move as two-heap {\sc nim}. Thus, the two round pieces have values $\nimber{(2 \oplus 1)} = \nimber{3}$ and $\nimber{(4 \oplus 5)} = \nimber{1}$.

The triangle piece has the same move as an impartial loopy ruleset 3-{\sc keep-nim} introduced in \cite{ASS24}. Thus, position $(x,y)$ has the value
$$
\left \{ \begin{array}{cc}
\nimber{(x \oplus y)}  & (x \oplus y \leq 2) \\ 
   \infty(\{0, 1 ,2\}) & (x \oplus y \geq 3).
\end{array}\right.
$$

Therefore, the two triangle pieces has values $\nimber{(3\oplus 3)} = 0$ and $\infty (\{0, 1, 2\})$.

The square piece has the same move as two-heap {\sc nim} with carry-on move when one moves to a position whose number of all  stones is three. We call this ruleset 3-{\sc turn-keep-nim}.

The values are as Table \ref{tab:3teban} and it is determined by the following theorem.
\begin{theorem}
A position in 3-{\sc turn-keep-nim} $(x,y)$ has the value
$$
\left \{ \begin{array}{cc}
\nimber{(x \oplus y)}  & (x +  y \leq 3) \\
   \spmoon(3) & (x+y>3, \min(\{x,y\}) \leq 3) \\
   \nimber{((x-4) \oplus (y-4))} & (\min(\{x,y\}>3).
\end{array}\right.
$$.
\end{theorem}
\begin{proof}
    We prove by induction on $x+y$.
    
    For the cases $x+ y\leq 3$, we can confirm individually.
    
    Assume that $x +y> 3$ and $\min(\{x, y\} \leq  3$. Then Left can move a position $(x, y) (x+y= 3)$, and choose whether she move again or end her move. Thus, Left has both $\nimber{3}$ and $\{ \infty \mid \nimber{3}\}$ as left options. Similarly, Right has both $\nimber{3}$ and $\{\nimber{3} \mid \minfty\}$ as Right options. Therefore, the value of the position is $\spmoon(3)$.
    
    When $\min(\{x,y\}) > 3$, by induction, options from $(x-4, y-4)$ of standard two-heap {\sc nim} and $\spmoon(3)$ are the options of the position. Therefore,  the value is $\nimber{((x-4)\oplus(y-4))}$. 
\end{proof}

From this theorem, the two square pieces have values $\spmoon(3)$ and $\nimber{((4-4) \oplus (5-4))} = \nimber{1}$.

\begin{table}[tb]
    \centering
    \begin{tabular}{c|ccccccccc}
         & 0 & 1 & 2 & 3 & 4 & 5 & 6 & 7 &  $\cdots$ \\ \hline
        0 & 0 & 1 & 2 & 3 & $\spmoon(3)$ & $\spmoon(3)$ & $\spmoon(3)$ & $\spmoon(3)$ & $\cdots$ \\
        1 & 1 & 0 & 3 & $\spmoon(3)$ & $\spmoon(3)$ & $\spmoon(3)$ & $\spmoon(3)$ & $\spmoon(3)$ & $\cdots$ \\
        
        2 & 2 & 3 & $\spmoon(3)$ & $\spmoon(3)$ & $\spmoon(3)$ & $\spmoon(3)$ & $\spmoon(3)$ & $\spmoon(3)$ & $\cdots$ \\
        3 & 3 & $\spmoon(3)$ & $\spmoon(3)$ & $\spmoon(3)$ & $\spmoon(3)$ & $\spmoon(3)$ & $\spmoon(3)$ & $\spmoon(3)$ & $\cdots$ \\
        4 & $\spmoon(3)$ & $\spmoon(3)$ & $\spmoon(3)$ & $\spmoon(3)$ & 0 & 1 & 2 & 3 & $\cdots$ \\
         5 & $\spmoon(3)$ & $\spmoon(3)$ & $\spmoon(3)$ & $\spmoon(3)$ & 1 & 0 & 3 & 2 & $\cdots$ \\
         6 & $\spmoon(3)$ & $\spmoon(3)$ & $\spmoon(3)$ & $\spmoon(3)$ & 2 & 3 & 0 & 1 & $\cdots$ \\
         7 & $\spmoon(3)$ & $\spmoon(3)$ & $\spmoon(3)$ & $\spmoon(3)$ & 3 & 2 & 1 & 0 & $\cdots$ \\
         $\vdots$ & $\vdots$ & $\vdots$ & $\vdots$ & $\vdots$ & $\vdots$ & $\vdots$ & $\vdots$ & $\vdots$ & $\ddots$
                
    \end{tabular}
    \caption{Values in {\sc 3-turn-keep-nim}.}
    \label{tab:3teban}
\end{table}

Therefore, in total, the position in Figure \ref{fig:trpiece} is 
$$
\nimber{3} + \nimber{1} + 0 + \infty(\{0, 1, 2\}) + \spmoon(3) + \nimber{1}.
$$
.

By using Table \ref{tab:elimpsummary}, we can calculate 
\begin{eqnarray}
    \nimber{3} + \nimber{1} + 0 + \infty(\{0, 1, 2\}) + \spmoon(3) + \nimber{1}. & \eqeq & \nimber{2} + \infty(\{0, 1, 2\}) + \spmoon(3) + \nimber{1} \nonumber \\ \nonumber
    & \eqeq &   \infty(\{2, 3, 0\}) + \spmoon(3) + \nimber{1} \\ \nonumber
    & \eqeq & \infty(\{1, 0, 3\})  + \nimber{1} \\ \nonumber
    & \eqeq & \infty(\{0, 1, 2\}). \nonumber
\end{eqnarray}
Thus, the position is an $\mathcal{N}$-position.

\section{Conclusion}

In this study, we consider the sum of impartial entailing games and impartial loopy games. We showed the structure of them and how to determine the outcome of such positions.  Since in this study we restrict the set of positions of entailing moves, in future study, we would like to remove the restriction and reveal the algebraic construction of the sum of any impartial entailing games and impartial loopy games.
We also would like to consider positions which has both impartial entailing positions and impartial loopy positions in their options.

\section*{Acknowledgment}

The author is really grateful to Professor Ko Sakai and Doctor Tomoaki Abuku for their valuable comments.


\begin{thebibliography}{99}
\bibitem{ASS24}
Abuku T., Sakai K., Suetsugu K.: The World of Combinatorial Game Theory: Winning Strategy with Mathematics, KYORITSU SHUPPAN (2024, in Japanese).

\bibitem{ANW11}
 Albert M. H., Nowakowski R. J., Wolfe D.: Lessons in Play: An Introduction to Combinatorial Game Theory, A K Peters/ CRC Press (2007).

\bibitem{FT75}
Fraenkel A. S., Tassa U.: Strategy for a class of games with dynamic ties, Comput. Math. Appl. {\bf 1}, pp.237-254 (1975).

\bibitem{FP73} Fraenkel A. S., Perl Y.: Constructions in combinatorial games with cycles, Infinite and Finite Sets, Vol.2 (A. Hajnal, R. Rado, and V. T. S\'{o}s, eds.), Colloq. Math. Soc. J\'{a}nos Bolyai, no. 10, North-Holland, pp.667-699 (1973).

\bibitem{FY86}
Fraenkel A. S., Yesha Y., The generalized Sprague-Grundy function and its invariance under certain mappings, J. Combin. Theory, Ser. A. {\bf 43} pp.165-177 (1986).

\bibitem{Gru39}
Grundy P.~M.: Mathematics and games, {\em Eureka},  Vol.~2, pp.\ 6--9 (1939).

\bibitem{LNS21}
Larsson U., Nowakowski R. J., Santos C. P.:
Impartial games with entailing moves, Integers, vol.21B, A17 (2021).

\bibitem{LNS23}
Larsson U., Nowakowski R. J., Santos C. P.:
A complete solution for a nontrivial ruleset with entailing moves, 
arXiv:2304.00588 [math.CO](2023).

\bibitem{LNS24}
Larsson U., Nowakowski R. J., Santos C. P.: 
Affine normal play, arXiv:2402.057	32 [math.CO](2024).

\bibitem{Sie13}
 Siegel A. N.:
{\it Combinatorial Game Theory},
American Mathematical Society(2013). 

\bibitem{Smi66}
Smith C. A. B.: Graphs and composite games, J. Combin. Theory, Ser. A {\bf 1} pp.51-81(1966).

\bibitem{Spr35}
Sprague R.~P.: \"{U}ber mathematische Kampfspiele, {\em T\^{o}hoku Math. J.},
  Vol.~41, pp.\ 438--444 (1935-36).
  
\end{thebibliography}
\end{document}